\documentclass[12pt]{article}
\usepackage{graphicx,subfigure}
\usepackage{amsmath,amsthm,amssymb,enumerate}
\usepackage{euscript,mathrsfs}
\usepackage{color}
\usepackage{dsfont}
\usepackage{url}
\usepackage{bm}
\usepackage[left=2cm,right=2cm,top=3.5cm,bottom=3.5cm]{geometry}
\usepackage{color}
\usepackage[framemethod=tikz]{mdframed}
\allowdisplaybreaks

\usepackage{soul}

\catcode`\@=11 \@addtoreset{equation}{section}

\catcode`\@=12

\newtheorem{Theorem}{Theorem}[section]
\newtheorem{Proposition}[Theorem]{Proposition}
\newtheorem{Lemma}[Theorem]{Lemma}
\newtheorem{Corollary}[Theorem]{Corollary}

\theoremstyle{definition}
\newtheorem{Definition}[Theorem]{Definition}

\newtheorem{Remark}[Theorem]{Remark}

\newcommand{\bTheorem}[1]{
	\begin{Theorem} \label{T#1} }
	\newcommand{\eT}{\end{Theorem}}

\newcommand{\bProposition}[1]{
	\begin{Proposition} \label{P#1}}
	\newcommand{\eP}{\end{Proposition}}

\newcommand{\bLemma}[1]{
	\begin{Lemma} \label{L#1} }
	\newcommand{\eL}{\end{Lemma}}

\newcommand{\bCorollary}[1]{
	\begin{Corollary} \label{C#1} }
	\newcommand{\eC}{\end{Corollary}}

\newcommand{\bRemark}[1]{
	\begin{Remark} \label{R#1} }
	\newcommand{\eR}{\end{Remark}}

\newcommand{\bDefinition}[1]{
	\begin{Definition} \label{D#1} }
	\newcommand{\eD}{\end{Definition}}

\newcommand{\YM}[1]{ \left< \mathcal{V}_{t,x}; #1 \right> }

\newcommand{\tvm}{\widetilde{\vc{m}}}
\renewcommand{\tvm}{\widetilde{\bm{m}}}
\newcommand{\tS}{\widetilde{S}}
\newcommand{\bfphi}{\boldsymbol{\varphi}}

\newcommand{\bFormula}[1]{
	\begin{equation} \label{#1}}
	\newcommand{\eF}{\end{equation}}

\newcommand{\Ov}[1]{\overline{#1}}

\newcommand{\toMo}{\stackrel{\mathcal{M}}{\to}}
\newcommand{\toW}{\stackrel{\mathcal{W}}{\to}}

\newcommand{\vr}{\varrho}

\newcommand{\tvr}{\wtilde \vr}

\newcommand{\tvt}{\wtilde \vt}

\newcommand{\vt}{\vartheta}
\newcommand{\vu}{\bm{u}}
\newcommand{\vm}{\bm{m}}

\newcommand{\vc}[1]{{\bf #1}}

\newcommand{\Div}{{\rm div}_x}
\newcommand{\Grad}{\nabla_x}

\newcommand{\dx}{\,{\rm d} {x}}

\newcommand{\dt}{\,{\rm d} t }

\newcommand{\intO}[1]{\int_{\Omega} #1 \ \dx}

\newcommand{\D}{{\rm d}}

\newcommand{\ep}{\varepsilon}

\newcommand{\br}{ \nonumber \\ }

\def\softd{{\leavevmode\setbox1=\hbox{d}%
		\hbox to 1.05\wd1{d\kern-0.4ex{\char039}\hss}}}
\definecolor{Cgrey}{rgb}{0.85,0.85,0.85}
\definecolor{Cblue}{rgb}{0.50,0.85,0.85}
\definecolor{Cred}{rgb}{1,0,0}
\definecolor{fancy}{rgb}{0.10,0.85,0.10}
\definecolor{amaranth}{rgb}{0.9, 0.17, 0.31}

\newcommand\Cbox[2]{%
	\newbox\contentbox%
	\newbox\bkgdbox%
	\setbox\contentbox\hbox to \hsize{%
		\vtop{
			\kern\columnsep
			\hbox to \hsize{%
				\kern\columnsep%
				\advance\hsize by -2\columnsep%
				\setlength{\textwidth}{\hsize}%
				\vbox{
					\parskip=\baselineskip
					\parindent=0bp
					#2
				}%
				\kern\columnsep%
			}%
			\kern\columnsep%
		}%
	}%
	\setbox\bkgdbox\vbox{
		\color{#1}
		\hrule width  \wd\contentbox %
		height \ht\contentbox %
		depth  \dp\contentbox
		\color{black}
	}%
	\wd\bkgdbox=0bp%
	\vbox{\hbox to \hsize{\box\bkgdbox\box\contentbox}}%
	\vskip\baselineskip%
}

\mdfdefinestyle{MyFrame}{%
	linecolor=black,
	outerlinewidth=1pt,
	roundcorner=5pt,
	innertopmargin=\baselineskip,
	innerbottommargin=\baselineskip,
	innerrightmargin=10pt,
	innerleftmargin=10pt,
	backgroundcolor=white!20!white}

\newcommand{\wtilde}{\widetilde}

\allowdisplaybreaks

\begin{document}


\title{\bf Well-posedness of the Euler system of gas dynamics}

\author{Eduard Feireisl
	\thanks{The work of E.F. was partially supported by the
		Czech Sciences Foundation (GA\v CR), Grant Agreement
		24--11034S. The Institute of Mathematics of the Academy of Sciences of
		the Czech Republic is supported by RVO:67985840.
		E.F. is a member of the Ne\v cas Center for Mathematical Modelling and Mercator Fellow in SPP 2410 ``Hyperbolic Balance Laws: Complexity, Scales and Randomness".}
	\and
M\' aria Luk\'a\v{c}ov\'a-Medvi\softd ov\'a
\thanks{The work of  M.L.-M. was supported by the Gutenberg Research College and by
		the Deutsche Forschungsgemeinschaft (DFG, German Research Foundation) -- project number 233630050 -- TRR 146 and
		project number 525853336-- SPP 2410 ``Hyperbolic Balance Laws: Complexity, Scales and Randomness".
		She is also grateful to the  Mainz Institute of Multiscale Modelling for supporting her research. }  
}

\date{}

\maketitle

\centerline{$^*$Institute of Mathematics of the Academy of Sciences of the Czech Republic}
\centerline{\v Zitn\' a 25, CZ-115 67 Praha 1, Czech Republic}
\centerline{feireisl@math.cas.cz}

\medskip

\centerline{$^\dagger$Institute of Mathematics, Johannes Gutenberg-University Mainz}
\centerline{Staudingerweg 9, 55 128 Mainz, Germany}
\centerline{lukacova@uni-mainz.de}

\begin{abstract}
	
We propose a new two-step selection criterion applicable 
to the dissipative measure--valued solutions of the Euler system of 
gas dynamics. The process consists of a successive maximisation of the entropy 
production rate and the total energy defect, i.e. maximisation of the turbulent energy.
\begin{itemize}
\item If the selected solution is a weak solution of the Euler system,  then it is identified in the first step. Solutions selected in the second step are 
truly measure--valued maximising the energy defect. Accordingly, they are called turbulent solutions.	 
\item The energy defect of turbulent solutions vanishes with growing time.

\item The selected solutions depend in a Borel--measurable way on the initial data. In particular, they are ``\emph{almost continuously dependent}'' on the initial data.
\end{itemize}

\noindent	Finally, we discuss a selection criterion yielding a maximal dissipative solution of the Euler system in a single step.

\end{abstract}


{\small

\noindent
{\bf 2020 Mathematics Subject Classification:} 35 Q 31, 35 A 02, 35 D 30

\medbreak
\noindent {\bf Keywords:} Euler system of gas dynamics, generalised solutions, consistent approximation, maximal entropy production, admissibility

\tableofcontents

}

\section{Introduction}
\label{i}

The Euler system of gas dynamics governing the motion of 
a gas confined to a bounded domain 
$\Omega \subset R^d$, $d=1,2,3$ consists of three 
field equations (balance laws):
\begin{align}
	\partial_t \vr + \Div \vm &= 0, \label{i1}\\
	\partial_t \vm + \Div \left( \frac{\vm \otimes \vm}{\vr} \right) +
	\Grad p &= 0, \label{i2} \\
	\partial_t E + \Div \left[ \Big( E + p \Big) \frac{\vm}{\vr} \right] &= 0,
	\label{i3}	
\end{align}
where $\vr = \vr(t,x)$ is the mass density, $\vm = \vm(t,x) = \vr \vu (t,x)$ the momentum with the velocity $\vu$,
$p$ the pressure, and $E$ the total energy. In addition, we impose the impermeability condition
\begin{equation} \label{i4}
	\vm \cdot \bm{n}|_{\partial \Omega} = 0,\ \bm{n} \ \mbox{the outer normal vector to}\ \partial \Omega.
\end{equation}

As  is well-known, the classical solutions of the Euler 
system develop singularities - shock waves - for a fairly general class of initial data, while the problem admits infinitely many weak (distributional) solutions. What is more suprising is the fact that the ill-posedness in the class of weak solutions persists even if the original system is supplemented by an \emph{admissibility condition} 
\begin{equation} \label{i5}
	\partial_t S + \Div \left( S \frac{\vm}{\vr} \right)  \geq 0,
\end{equation}	
where $S$ is the total entropy of the system. More specifically, the recent results based on the application of convex integration include: 
\begin{itemize} 
	\item The existence of infinitely many admissible weak solutions for the system 
	\eqref{i1}--\eqref{i4} for a vast class of initial data, see \cite{FeKlKrMa}, \cite{KlKrMaMa}. 
\item Density of the so-called wild data giving rise to infinitely many admissible weak solutions, see \cite{ChFe2023}.
\item Ill-posedness of the problem perturbed by a stochastic forcing, see \cite{ChFeFl2019}.
\end{itemize}

\subsection{Two-step selection process}

The leading idea behind the concept of \emph{turbulent solution} goes back to \cite{BreFeiHof19C}, where a multilevel selection process was proposed to identify 
a unique element in the class of \emph{dissipative measure--valued} (DMV) solutions of the Euler system. The DMV solutions can be seen as limits of suitable consistent approximations arising as artificial viscosity limits or 
asymptotic limits of stable and consistent numerical schemes, cf.~the monograph \cite{FeLMMiSh}.

Similarly to \cite{BreFeiHof19C}, the turbulent solutions of the Euler system \eqref{i1}--\eqref{i5} will be represented by the trio $(\vr, \vm, S)$ supplemented with the \emph{total energy} 
\begin{equation} \label{i6}
	\mathcal{E}_0 = \intO{ E_0 }, 
\end{equation}	
where $E_0 = E(\vr_0, \vm_0, S_0) $ is the energy of the system at the initial state $(\vr_0, \vm_0, S_0)$.

In accordance with the boundary condition \eqref{i4}, the system is energetically \emph{closed}. Our approach is based on enforcing two underlying physical principles:

\begin{itemize}
	\item {\bf The First Law of Thermodynamics.} The total energy  $\mathcal{E} = \mathcal{E}_0$ is constant in time. 
	\item {\bf The Second Law of Thermodynamics.} The total 
	entropy $\intO{ S(t, \cdot) }$ is a non--decreasing function of time.
	 
\end{itemize}

The turbulent solutions will be selected among all 
DMV solutions to the Euler system emanating from given initial data. A prominent role in the selection process is played by the \emph{mean energy} 
\begin{equation} \label{i7}
\intO{ E(\vr, \vm, S) (t, \cdot) }, \ \mbox{where}\ 
E(\vr, \vm, S) = \frac{1}{2} \frac{|\vm|^2}{\vr} + \vr e(\vr, S).
\end{equation}
Unlike the total energy $\mathcal{E}$, the mean energy may not be constant in time but satisfies the inequality 
\begin{equation} \label{i8} 
\mathcal{D}(\vr, \vm, S)(t) = \mathcal{E} - \intO{E(\vr, \vm, S) (t, \cdot) } \geq 0,
\end{equation} 
where the quantity $\mathcal{D}(t)$  
is called \emph{energy defect}. 
The energy defect can be identified with \emph{turbulent energy} - the difference between the total and the mean energy of the system.
As we shall see, the 
standard \emph{weak solutions} of the Euler system are characterized by the property $\mathcal{D} \equiv 0$. The turbulent solutions, in general, will maximise the energy defect, in particular, $\mathcal{D}(t) > 0$. 

Let 
\[
\mathcal{U}[\vr_0, \vm_0, S_0] 
\]  
denote the family of all DMV solutions emanating from the 
initial data $(\vr_0, \vm_0, S_0)$ defined on the full time interval $t \in [0, \infty)$. Following the strategy of 
Krylov \cite{KrylNV}, developed in the context of deterministic problems by Cardona and Kapitanski \cite{CorKap}, and finally adapted to the Euler system in \cite{BreFeiHof19C}, we identify the unique solution via 
\emph{successive} maximisation/minimisation of two ``cost'' functionals: 
\begin{align}
\mathcal{F}_S[\vr, \vm, S] &=  \int_0^\infty \exp(-t) \left( \intO{ S(t, \cdot)} \right) \dt , \br  	
\mathcal{F}_E[\vr, \vm, S] &=  \int_0^\infty \exp(-t) \left( \intO{ E(\vr, \vm, S)(t, \cdot)} \right) \dt.
\label{i10}	
\end{align}
As the set $\mathcal{U}[\vr_0, \vm_0, S_0]$ is convex 
and $\mathcal{F}_S$ linear , the set of maximisers 
\[
\mathcal{U}_S[\vr_0, \vm_0, S_0] = {\rm argmax}\, \mathcal{F}_S \Big[ \mathcal{U}[\vr_0, \vm_0, S_0] \Big] 
\subset \mathcal{U}[\vr_0, \vm_0, S_0] 
\]
is a convex subset of $\mathcal{U}[\vr_0, \vm_0, S_0]$. 
Then the turbulent solution is obtained in the second minimisation step 
\[
\mathcal{U}_{ES}(\vr, \vm, S) = {\rm argmin}\, \mathcal{F}_E \Big[ {\rm argmax}\, \mathcal{F}_S [\mathcal{U}[\vr_0, \vm_0, S_0] 
 \Big].
\]
The fact that the second step yields a unique solution follows from the strict convexity of the functional $\mathcal{F}_E$. The strict convexity of $\mathcal{F}_E$, in turn, is equivalent to the hypothesis of thermodynamic stability imposed on the equations of state.

Similarly to \cite{FeiJuLu}, we have reduced the selection procedure to two steps. The idea of obtaining 
turbulent solutions via maximising the energy defect 
was also used by Klingenberg, Markfelder, and Wiedemann 
\cite{KlMaWi}. In the present context, this criterion yields a (unique) solution in only one step, minimising $\mathcal{F}_E$ directly. Unfortunately, such a selection would automatically eliminate all \emph{weak} solutions of the 
problem, except in a very unlikely event that there is a weak solution unique in the class of DMV solutions.  
Maximising the entropy production in the first step will select physically admissible solutions, including possibly weak solutions. In particular, we show that all of them comply with DiPerna's maximality criterion, cf. \cite{DiP2}. 
Moreover, their energy defect $\mathcal{D}(\vr, \vm, S)(t)$ vanishes asymptotically as $t \to \infty$.
The second step identifies the unique turbulent solution that maximises the energy defect—the turbulent energy. Thus, these two selection steps seem to go in opposite directions.

\subsection{Main objectives}

The two-step selection process described in the previous section yields 
the unique turbulent solution of the Euler system. Besides providing 
a detailed description of the selection process, we shall discuss the following issues. 

\begin{itemize} 
	\item {\bf Well-posedness.}
For any given initial data $(\vr_0, \vm_0, S_0)$, and the associated total energy $\mathcal{E}_0$, the Euler system admits a unique, turbulent, solution $(\vr, \vm, S)$. The mapping 
\[
(\vr_0, \vm_0, S_0) \mapsto (\vr, \vm, S) (t, \cdot) 
\] 
is Borel--measurable in suitable topologies. Consequently, applying the method. 
of \cite{BreFeiHof19C} we show the turbulent solutions form 
a Borel--measurable semigroup, see Section \ref{d}, and 
Theorem \ref{caT1}, Section \ref{step1}, Theorem \ref{caT2}, Section \ref{ds}.

\item {\bf Weak regularity.} 

If the selected solution is a \emph{weak} solution, then it must be identified in the first step, see Section \ref{step1}. Any solution selected in the second step is a truly turbulent solution with a non-zero energy defect  $\mathcal{D}$, see Section \ref{ds}.

\item {\bf Vanishing energy defect.}
The energy defect of any turbulent solution vanishes with growing time, 
\begin{equation} \label{i11}
	\mathcal{D}(\vr, \vm, S)(t) \to 0 \ \mbox{as}\ t \to \infty,
\end{equation}
see Section \ref{step1}.

\item {\bf Almost continuity.} 

Given $\ep > 0$ and a Borel probability measure on the space of initial data, the solution semigroup is continuous on a complement of a set of $\ep -$measure, see Section \ref{conc}. 

\end{itemize}

\section{Dissipative measure--valued (DMV) solutions}
\label{d}

For definiteness,
we consider a polytropic equation of state (EOS) 
\begin{equation} \label{r1a}
p = (\gamma - 1) \vr e ,\ \gamma > 1.
\end{equation}
Moreover, without loss of generality, we introduce the \emph{absolute temperature} $\vt$ proportional to internal energy $e$,
\begin{equation} \label{r2a}
e = c_v \vt, \ \mbox{where}\  c_v = \frac{1}{\gamma -1}.
\end{equation}
Accordingly, we set
\begin{equation} \label{r3a}
p(\vr, S) = (\gamma - 1) \vr e (\vr, S) = \left\{ \begin{array}{l}
\vr^\gamma \exp \left( \frac{S}{c_v \vr} \right) \ \mbox{if} \ \vr > 0, \\
0 \ \mbox{if}\ \vr = 0,\ S \leq 0, \\
\infty \ \mbox{otherwise}.
\end{array}	
\right.
\end{equation}
It is easy to check
\begin{equation} \label{r1}
	\frac{\partial \vr e(\vr, S)}{\partial S} = \vt > 0 \ \mbox{whenever}\ \vr > 0,
\end{equation}
meaning $\vr e(\vr, S)$ is a strictly increasing function of $S$.
Introducing the kinetic energy
\begin{equation} \label{RR1}
\frac{1}{2} \frac{|\vm|^2}{\vr} = \left\{ \begin{array}{l}
	\frac{1}{2} \frac{|\vm |^2}{\vr} \ \mbox{if}\ \vr > 0,\\
	0 \ \mbox{if}\ \vr = 0, \vm = 0, \\
	\infty \ \mbox{otherwise} \end{array}
\right.
\end{equation}
we can check that the total energy
\[
E(\vr, \vm, S) = \frac{1}{2} \frac{|\vm|^2}{\vr} + \vr e(\vr, S)
\]
is a convex l.s.c. function of the state variables $(\vm, \vr, S) \in R^{d+2}$, strictly convex on its domain. 
The convexity of the total energy is equivalent to   
\emph{thermodynamic stability} of the system, cf.
Bechtel, Rooney, Forest \cite{BEROFO} or
\cite[Chapter 2, Section 2.2.4]{FeLMMiSh}. This assumption
is absolutely crucial in the subsequent analysis. More general EOS can be considered as long as the thermodynamic stability is preserved.

\subsection{DMV solutions}

A dissipative measure-valued solution consists of a parametrised (Young) measure
\begin{equation} \label{d1}
	\mathcal{V} = \mathcal{V}_{t,x}: (t,x) \in (0,\infty) \times \Omega \to \mathfrak{P} (R^{d+2}),\
	\mathcal{V} \in L^\infty_{\rm weak-(*)}((0,\infty) \times \Omega; \mathfrak{P} (R^{d+2})),
\end{equation}
where $\mathfrak{P}$ denotes the set of Borel probability measures defined
on the space of ``dummy'' variables
\[
R^{d+2} = \left\{ (\tvr, \tvm, \widetilde{S}) \Big|  (\vr, \vm,  S) \in R^{d+2}             \right\},
\]
and a concentration measure
\begin{equation} \label{d1a}
\mathfrak{C} \in L^\infty_{\rm weak-(*)}(0,\infty; \mathcal{M}^+ (\Ov{\Omega}; R^{d \times d}_{\rm sym})),
\end{equation}
where
the symbol $\mathcal{M}^+ (\Ov{\Omega}; R^{d \times d}_{\rm sym}))$ denotes the set of all matrix-valued positively semi-definite finite Borel measures on
the compact set $\Ov{\Omega}$.

\medskip
	
\begin{Definition}[{\bf DMV solution}] \label{Dd2}
	
We shall say that a parametrized measure $\{ \mathcal{V}_{t,x} \}$ and
a concentration measure $\mathfrak{C}$ represent \emph{dissipative measure--valued (DMV) solution} of the Euler system \eqref{i1}--\eqref{i3}, \eqref{i4}, with the initial data $(\vr_0, \vm_0, S_0)$ and the initial total energy $\mathcal{E}_0$ if the 
following holds:
\begin{itemize}		
\item
The equation of continuity
\begin{equation} \label{d2}
\int_0^\infty \intO{ \Big[ \left< \mathcal{V}_{t,x}; \tvr \right> \partial_t \varphi(t,x) +
	\YM{ \widetilde{\vm} }\cdot \Grad \varphi(t,x) \Big]} \dt = - \intO{ \vr_0(x) \varphi(0, x) } 	
\end{equation}
holds for any $\varphi \in C^1_c([0,\infty) \times \Ov{\Omega})$.
\item
The momentum equation
\begin{align}
\int_0^\infty &\intO{ \left[ \YM{ \tvm } \cdot \partial_t \bfphi(t,x)  + \YM{ \mathds{1}_{\tvr > 0} \left(
\frac{ \tvm \otimes \tvm}{\tvr} + p(\tvr, \widetilde{S}) \mathbb{I} \right) } \Grad \bfphi(t,x) 	 \right] } \dt \br
&= - \int_0^\infty \int_{\Ov{\Omega}} \Grad \bfphi(t,x) : \D \mathfrak{C} (t,x)
\label{d3}	- \intO{ \vc{m}_0(x) \cdot \bfphi(0, x) }
	\end{align}
holds for any $\bfphi \in C^1_c([0,\infty) \times \Ov{\Omega}; R^d)$, $\bfphi \cdot \bm{n}|_{\partial \Omega} = 0$.
\item
The entropy inequality
\begin{align}
\int_0^\infty &\intO{ \left[ \YM{ \widetilde{S} } \partial_t \varphi(t,x) +
\YM{ \mathds{1}_{\tvr > 0} \widetilde{S} \frac{\tvm}{\tvr} } \cdot \Grad \varphi(t,x) \right] } \dt
\br &\leq - \intO{ S_0(x) \varphi (0, x) }
\label{d4}		
	\end{align}
holds for any $\varphi \in C^1_c([0,\infty) \times \Ov{\Omega})$, $\varphi \geq 0$.
\item
The energy compatibility condition
\begin{align}
	\mathcal{E}_0 \geq 
	\intO{ E(\vr_0, \vm_0, S_0) }
	&\geq \intO{ \YM{ E(\tvr, \tvm, \widetilde{S}) }  }	+   r(d, \gamma)\int_{\Ov{\Omega}} \ \D {\rm trace}[ \mathfrak{C}](t,x) ,\br
	r(d, \gamma) &= \min \left\{ \frac{1}{2}; \frac{d \gamma}{\gamma - 1} \right\}
	\label{d5}
\end{align}	
holds for a.a. $t \in (0,\infty)$.

\end{itemize}

\end{Definition}

\medskip

The concept of DMV solution was introduced in \cite{BreFei17} and later elaborated in the monograph \cite[Chapter 5]{FeLMMiSh}. Kr\" oner and Zajaczkowski \cite{KrZa}
introduced a different class of measure--valued solution to the Euler system postulating entropy \emph{equation} rather than inequality \eqref{d4}. Unfortunately,
a physically admissible weak solution with a shock that produces entropy does not fit in this class.

Let us summarise some of the known properties of DMV solutions. 

\begin{enumerate}
	\item {\bf Global existence.} Let the initial data belong to the class 
\begin{align}
	(\vr_0, \vm_0, S_0) \in L^1(\Omega; R^{d+2}),\
	\intO{ E(\vr_0, \vm_0, S_0) } \leq \mathcal{E}_0 < \infty, \ S_0 \geq \underline{s} \vr_0 	\
	\mbox{a.a. in}\ \Omega,
	\label{dclass}
\end{align}
where $\underline{s} \in R$ is a given constant. Then the Euler system \eqref{i1}--\eqref{i3}, \eqref{i4} admits a DMV solution
in $(0,\infty) \times \Omega$ specified in Definition \ref{Dd2}. In addition,
the solution satisfies
\begin{equation}\label{dd1}
	\mathcal{V}_{t,x} \left\{ \tvr \geq 0,\ \widetilde{S} \geq \underline{s} \tvr \right\} = 1 \ \mbox{for a.a.}\ (t,x) \in (0,\infty) \times \Omega.
\end{equation}
For the proof see \cite[Proposition 3.8]{BreFeiHof19C}. 

In addition, the expected values 
\begin{equation} \label{expe}
\vr(t,x) = \left< \mathcal{V}_{t,x}; \tvr \right>,\ 
\vm (t,x) = \left< \mathcal{V}_{t,x}; \tvm \right>, 
S(t,x) = \left< \mathcal{V}_{t,x}; \widetilde{S} \right>\ 
\end{equation}	
belong to the class
\begin{equation} \label{expe1}
\vr, S \in L^\infty(0,\infty; L^\gamma (\Omega)),\ \vm \in L^\infty(0,\infty; L^{\frac{2 \gamma}{\gamma + 1}}(\Omega; R^d)),
\end{equation}	
see \cite[Chapter 4, Section 4.1.5]{FeLMMiSh}.	
Similarly to \cite{BreFeiHof19C}, 
we call the triple of expected values $(\vr, \vm, S)$, together with $\mathcal{E}_0$,  
\emph{dissipative solution} of the Euler system.
\item {\bf Convexity and compactness.} 
Let 
\begin{align} \label{expe3}
\mathcal{U}[\vr_0, \vm_0, S_0; \mathcal{E}_0] = 
&\left\{ (\vr, \vm, S) \ \Big|\ (\vr, \vm, S) 
\ \mbox{is a dissipative solution in}\ (0,\infty) \times 
\Omega \right. \br
&\mbox{emanating from the data}\ 
(\vr_0, \vm_0, S_0; \mathcal{E}_0) \Big\}
\end{align}
denote the solution set associated to the data $(\vr_0, \vm_0, S_0; \mathcal{E}_0)$

Then the following holds:
\begin{itemize}
	\item The set $\mathcal{U}[\vr_0, \vm_0, S_0; \mathcal{E}_0]$ is non--empty and convex for any initial data satisfying 
	\eqref{dclass}. 
	\item The set $\mathcal{U}[\vr_0, \vm_0, S_0; \mathcal{E}_0]$ is closed
	and bounded  
	in the weighted Lebesgue space 
	\begin{equation} \label{expe4}
	L^q_\omega ((0,\infty) \times \Omega; R^{d+2}) \ \mbox{with the measure} \ \D \omega = \exp(-t)\dt \times \dx 
	 \end{equation}
for any $1\leq q \leq \frac{2 \gamma}{\gamma + 1}$. 
\item The set $\mathcal{U}[\vr_0, \vm_0, S_0; \mathcal{E}_0]$ is compact in the space
\begin{equation} \label{expe5}
L^q_\omega(0, \infty; W^{-\ell,2}(\Omega; R^{d+2}) ),\ \ell > 3,	
\end{equation}
for any $1\leq q < \infty$.		 
\end{itemize}
For the proof, see \cite[Chapter 5, Theorem 5.2]{FeLMMiSh}.

\item {\bf Continuity in time.}

The dissipative solution $(\vr, \vm, S)$ admits 
\emph{instantaneous values} defined 
\begin{align}
\lim_{\delta \to 0+} \frac{1}{\delta}\int_{t - \delta}^t\intO{ \vr(s, x) \varphi(x) } \D s
&= \intO{ \vr(t,x) \varphi(x) } =  
	\lim_{\delta \to 0+} \frac{1}{\delta} \int_{t}^{t +\delta} \int_\Omega \vr(s,x) \varphi(x) \mbox{d} x  \mbox{d} s, \br
\lim_{\delta \to 0+} \frac{1}{\delta}\int_{t - \delta}^t\intO{ \vm(s, x) \cdot \bfphi(x) } \D s
&= \intO{ \vm(t,x) \cdot \bfphi(x) } \br &=  
\lim_{\delta \to 0+} \frac{1}{\delta} \int_{t}^{t +\delta} \int_\Omega \vm(s,x) \cdot \bfphi(x) \mbox{d} x  \mbox{d} s 
\nonumber 
\end{align}
for any $\varphi \in L^\infty(\Omega)$, $\bfphi \in L^\infty(\Omega; R^d)$; and
\begin{align}
\lim_{\delta \to 0+} \frac{1}{\delta} \int_{t - \delta}^t 
\intO{ S(s,x) \varphi(x) } \D s &\leq 
\intO{ S(t-,x) \varphi (x) } \br &\leq 
\intO{ S(t+,x) \varphi(x) }	= 	
	\lim_{\delta \to 0+} \frac{1}{\delta} \int_{t}^{t + \delta} \int_\Omega S(s,x) \varphi(x) \mbox{d} x \mbox{d} s,
	\nonumber
\end{align}
for any $\varphi \in L^\infty(\Omega)$, $\varphi \geq 0$ provided we have set
\[
\vr(t,\cdot) = \vr_0, \ \vm(t,\cdot) = \vm_0,\ S(t,\cdot) = S_0 \ \mbox{for}\ t \leq 0.
\]

 In view of the above identities, we have
\begin{align}
	\vr &\in C_{\rm weak}([0,\infty); L^\gamma (\Omega)), \ \vr(0, \cdot) = \vr_0, \br
	\vm &\in C_{\rm weak}([0, \infty; L^{\frac{2 \gamma}{\gamma + 1}}(\Omega; R^d)),\
	\vm(0, \cdot) = \vm_0, \br
	S &\in BV_{\rm weak}(0, \infty; L^\gamma (\Omega)),\
	S(0-, \cdot) = S_0,\ S(t-, \cdot) \leq S(t+, \cdot) \ \mbox{for any}\ t \geq 0
	\label{dd2}
\end{align}	
see \cite{BreFeiHof19C} for details. In particular, the entropy 
component of a dissipative solution can be identified with the left limit $S(t-, \cdot)$, which can be seen as a (weakly) c\` agl\` ad function of $t \in [0,\infty)$ with values in $L^\gamma(\Omega)$. Finally, it follows from the entropy inequality \eqref{d4} that the total entropy
\[
t \in [0, \infty) \to \intO{ S(t-, \cdot) },\ S(0-, \cdot) = S_0
\]
is a non--decreasing (c\` agl\` ad) function in $[0, \infty)$. 

\item {\bf Weak measurability.}
Motivated by \cite{BreFeiHof19C} we introduce the data space 
\begin{align} \
	X_D &= \left\{ (\vr_0, \vm_0, S_0; \mathcal{E}_0)     \Big| 
	(\vr_0, \vm_0, S_0) \in L^1(\Omega; R^{d+2}),\ 
	S_0 \geq \underline{s} \vr_0 ,\ 
	\intO{ E(\vr_0, \vm_0, S_0) } \leq \mathcal{E}_0 \right\} 
	\br &\subset W^{-\ell,2}(\Omega; R^{d+2}) \times [0, \infty), \ \ell > d,
	\label{ca1}
\end{align}
with the topology of the Hilbert space $W^{-\ell,2}(\Omega; R^{d+2}) \times R$. Given $\underline {s} \in R$, the data space $X_D$ is a non--empty convex compact subset of the Hilbert space
$W^{-\ell,2}(\Omega; R^{d+2}) \times R$. 

As shown 
in \cite[Section 4, Lemma 4.5]{BreFeiHof19C}, the solution mapping 
\begin{equation} \label{ca2}
	\mathcal{U}: (\vr_0, \vm_0, S_0; \mathcal{E}_0) \in X_D \mapsto 
	\mathcal{U}[\vr_0, \vm_0, S_0; \mathcal{E}_0] \in  {\rm comp} \Big[ L^2_\omega (0,\infty; W^{-\ell,2} (\Omega; R^{d+2})) \Big]  , 
\end{equation}	
where ${\rm comp}$ denotes the metric space of all compact subsets 
endowed with the Hausdorff topology, is Borel--measurable. Moreover, 
\begin{equation} \label{ca33}
(\vr(t, \cdot), \vm(t, \cdot), S(t\pm), \cdot), \mathcal{E}_0) 
\in X_D	\ \mbox{for any}\ t \geq 0
\end{equation}	
whenever 
\[
(\vr, \vm, S)  \in \mathcal{U}[\vr_0, \vm_0, S_0; \mathcal{E}_0]. 
\]

\end{enumerate}

Besides the above facts crucial for the subsequent 
analysis, the dissipative solutions enjoy other important properties, 
among which the weak--strong uniqueness and compatibility with the strong solutions, see \cite{FeLMMiSh}.  

\section{Step 1 - maximal entropy production}
\label{step1}

Following the strategy of \cite{BreFeiHof19C} we introduce the 
\emph{first selection functional} 
\[
\mathcal{F}_S (\vr, \vm, S) = \int_0^\infty  \exp(-t) \left( \intO{ S(t, \cdot) } \right) 
\dt 
\]
which is a bounded linear form defined on the solution set 
$\mathcal{U}[\vr_0, \vm_0, S_0; \mathcal{E}_0]$ endowed with the topology of the Banach space $L^q_\omega ((0,\infty) \times \Omega; R^{d+2}) \times R$, $1 \leq q \leq \frac{2 \gamma}{\gamma +1}$.

Next, let us introduce the set of maximisers of the entropy production, 
\begin{align}
\mathcal{U}_S[ \vr_0, \vm_0, S_0; \mathcal{E}_0] &\equiv	{\rm argmax} \mathcal{F}_S \Big[ \mathcal{U}[\vr_0, \vm_0, S_0; \mathcal{E}_0] \Big]\br &= \left\{ (\vr, \vm, S) \in \mathcal{U}[\vr_0, \vm_0, S_0; \mathcal{E}_0]\ \Big|\ \mathcal{F}_S [\vr, \vm, S] \geq 
	\mathcal{F}_S [\tvr, \tvm, \widetilde{S}]  \right. \br 
	&\quad \quad \mbox{for any}\ (\tvr, \tvm, \widetilde{S}) \in 
	\mathcal{U}[\vr_0, \vm_0, S_0; \mathcal{E}_0]    \Big\}. 
	\label{ca3}
\end{align}
We refer to \cite[Section 4, Lemma 4.5]{BreFeiHof19C} for the proof that the set valued mapping
\begin{equation} \label{ca4}
	\mathcal{U}_S: (\vr_0, \vm_0, S_0; \mathcal{E}_0) \in X_D \mapsto 
	{\rm argmax}\, \mathcal{F}_S \Big[ \mathcal{U}[\vr_0, \vm_0, S_0; \mathcal{E}_0] \Big] \in {\rm comp} \Big[ L^2_\omega (0,T; W^{-\ell,2} (\Omega; R^{d+2})) \Big]
\end{equation}	 
is Borel--measurable. Moreover, as $\mathcal{F}_S$ is linear, the 
set $\mathcal{U}_S [ \vr_0, \vm_0, S_0; \mathcal{E}_0] $ is a convex set. 

Our next goal is to show the Borel measurability of the same mapping 
\begin{align} 
	\mathcal{U}_S: (\vr_0, \vm_0, S_0; \mathcal{E}_0) \in X_D \mapsto 
	{\rm argmax}\, \mathcal{F}_S \Big[ \mathcal{U}[\vr_0, \vm_0, S_0; \mathcal{E}_0] \Big] &\in {\rm closed} \Big[ L^q_\omega (0,\infty; L^q (\Omega; R^{d+2})) \Big],\br 1 < q &\leq \frac{2 \gamma}{\gamma + 1},
	\label{ca5} 
\end{align}	
where the symbol ${\rm closed}$ refers to the family of 
all closed subsets. 
Indeed the set-valued mapping $\mathcal{U}_S$ specified in 
\eqref{ca4} is Borel--measurable, ranging in a compact metric space. Accordingly, it admits the so-called Castaign representation - 
a countable family of Borel--measurable selections - 
\begin{align}
\mathcal{U}^i_{S}: (\vr_0, \vm_0, S_0; \mathcal{E}_0) &\mapsto \br 
(\vr^i, \vm^i, S^i) &\in {\rm argmax} \mathcal{F}_S \Big[ \mathcal{U}[\vr_0, \vm_0, S_0; \mathcal{E}_0] \Big]
\in {\rm comp} \Big[ L^2_\omega (0,\infty; W^{-\ell,2} (\Omega; R^5)) \Big], 
\nonumber
\end{align}
$i=1,2, \dots$ such that
\[
{\rm closure}_{ L^2_\omega (0,T; W^{-\ell,2} (\Omega; R^5))}
\Big\{ (\vr^i, \vm^i, S^i)_{i=1}^\infty   \Big\} 
= {\rm argmax} \mathcal{F}_S \Big[ \mathcal{U}[\vr_0, \vm_0, S_0; \mathcal{E}_0] \Big], 
\]
cf.~Castaing, Valadier \cite{CasVal}.
Now, we may extend the family $(\vr^i, \vm^i, S^i)_{i=1}^\infty$ by considering finite convex combinations 
with rational coefficients, 
\[
\sum_{j=1}^m \lambda_j (\vr^i, \vm^i, S^i)  ,\ \lambda_j \in Q
\]

As the set ${\rm argmax}\, \mathcal{F}_S \Big[ \mathcal{U}[\vr_0, \vm_0, S_0; \mathcal{E}_0] \Big]$ is convex, we have 
\begin{align}
{\rm closure}_{ L^2_\omega (0,T; W^{-\ell,2} (\Omega; R^5))}
\Big\{ \sum_{j=1}^m \lambda_j (\vr^i, \vm^i, S^i)_{i,m=1}^\infty   \Big\} 
&= {\rm argmax}\, \mathcal{F}_S \Big[ \mathcal{U}[\vr_0, \vm_0, S_0; \mathcal{E}_0] \Big] \br &= 
\mathcal{U}_S [\vr_0, \vm_0, S_0, \mathcal{E}_0].
\nonumber
\end{align}
Finally, as the weak and strong closures coincide on convex sets, the extended system can be seen as a Castaing representation of 
$\mathcal{U}_S$ in the topology of the space 
$L^q_\omega(0, \infty; L^q(\Omega; R^{d+2}))$. This yields the desired
Borel measurability of $\mathcal{U}_S$ specified in \eqref{ca5}.

\begin{Remark} \label{Rac1}
In view of Hess's measurability theorem \cite{Hess}, the measurability of the solution mapping is equivalent to the Borel measurability of the set-valued mapping.
\[
\mathcal{U}_S: X_D \to {\rm closed}[L^q_\omega(0,\infty; L^q(\Omega; R^{d+2})],
\]
where the set ${\rm closed}[L^q_\omega(0,\infty; L^q(\Omega; R^{d+2})]$ 
is endowed with the (metrizable) Wijsman topology: 
\begin{equation} \label{s11}
	\mathcal{A}_n \toW \mathcal{A} \ \Leftrightarrow\
	{\rm dist}_{L^q_\omega}[y, \mathcal{A}_n] \to {\rm dist}_{L^q_\omega}[y, \mathcal{A}] \ \mbox{for any}\ y \in L^q_\omega(0,\infty; L^q(\Omega; R^{d+2}) .
\end{equation}
	
\end{Remark}	

\subsection{Properties of the solutions selected in Step 1}

Solutions selected in Step 1 belong to the convex set of entropy maximisers. 
\[
\mathcal{U}_S[\vr_0, \vm_0, S_0; \mathcal{E}_0] = {\rm argmax}\, \mathcal{F}_S \Big[ 
\mathcal{U}[\vr_0, \vm_0, S_0; \mathcal{E}_0] \Big],
\]
supplemented with the associated total energy component $\mathcal{E}_0$.
First, we show that these solutions comply with the so-called DiPerna maximality criterion, cf.~\cite{DiP2}. To see this, we introduce a partial ordering on the set $\mathcal{U}[\vr_0, \vm_0, S_0; \mathcal{E}_0]$: 
\begin{align}
	(\vr^1, \vm^1, S^1) &\prec_{DiP} (\vr^2, \vm^2, S^2) \br
	\qquad &\Leftrightarrow_{\rm def}  \qquad \br 
	\intO{ S^1(t+, \cdot) } &\leq  	\intO{  S^2(t+, \cdot) } \ \mbox{for all}\ t \geq 0. 
	\nonumber
\end{align}	
We say that a solution $(\tvr, \tvm, \tS)$ is 
$\prec_{DiP}$ maximal in the set $\mathcal{U}[\vr_0, \vm_0, S_0; \mathcal{E}_0]$ if for any other solution $(\vr, \vm, S; \mathcal{E}_0) 
\in \mathcal{U}[\vr_0, \vm_0, S_0; \mathcal{E}_0]$
satisfying 
\[
(\tvr, \tvm, \tS) \prec_{DiP} (\vr, \vm, S)
\]
there holds	
\[
\intO{ \tS(t+, \cdot) } =  	\intO{ S(t+, \cdot) } \ \mbox{for all}\ t > 0.
\]

It is easy to check that all solutions selected in Step 1 belong to the set 
\[
\mathcal{U}_S [\vr_0, \vm_0, S_0; \mathcal{E}_0] = {\rm argmax}\, \mathcal{F}_S \Big[ 
\mathcal{U}(\vr_0, \vm_0, S_0; \mathcal{E}_0) \Big],
\]
are $\prec_{DiP}$ maximal in  $\mathcal{U}[\vr_0, \vm_0, S_0; \mathcal{E}_0]$.

DiPerna \cite{DiP2} conjectured that all 
$\prec_{DiP}$ maximal solutions are in fact weak solutions, meaning 
their energy defect 
\[
\mathcal{D}(\vr, \vm, S)(t) = \mathcal{E}_0 - \intO{ E(\vr, \vm, S)(t) }
\]
vanishes for a.a. $t > 0$. Although the conjecture remains open, we can show that the energy defect of any $\prec_{DiP}$ maximal solution 
vanishes for $t \to \infty$.

To show this, we firstly  we recall the concatenation property of DMV solutions proved in 	
\cite[Lemma 5.1]{FeiLukYu}: 
\begin{Lemma}[{\bf Concatenation of DMV solutions}] \label{Ldv1}
	
	Let $\{ \mathcal{V}^i_{t,x} \}, \mathfrak{C}^i$ be two DMV solutions of the Euler system in $(0, T_i) \times \Omega$, with the initial data
	$(\vr_0^i, \vm_0^i, S^i_0, \mathcal{E}_0)$, $i=1,2$ respectively. Let 
	\[
	\vr^i (t,x) = \left< \mathcal{V}^i_{t,x}, \tvr \right>,\  \vm^i (t,x) = \left< \mathcal{V}^i_{t,x}, \tvm \right>, \
	S^i(t,x) = \left< \mathcal{V}^i_{t,x}, \tS \right> ,\ i=1,2
	\]
	be the associated dissipative solutions.
	Suppose that
	\begin{equation} \label{cp1}
		\vr^1(T_1, \cdot) = \vr^2_0, \ \vm^1(T_1, 0) = \vm^2_0,  	
	\end{equation}
and	
	\begin{equation} \label{cp2}
		S^1(T_1-, \cdot)   \leq S^2_0.
	\end{equation}

Then the concatenated solution $(\{ \mathcal{V}_{t,x} \}, \mathfrak{C}) = (\{\mathcal{V}^1_{t,x}\} ,\mathfrak{C}_1)  \cup_{T_1} (\{ \mathcal{V}^2_{t,x} \}, \mathfrak{C}_2) $ defined as
	\[
	\mathcal{V}_{t,x} = \left\{ \begin{array}{l} \mathcal{V}^1_{t,x} \ \mbox{for}\ 0 \leq t \leq T^1, \\
		\mathcal{V}^2_{t - T^1,x} \ \mbox{for}\ T^1 < t \leq T^1 + T^2 \end{array} \right.	,\
	\mathfrak{C}(t) = \left\{ \begin{array}{l} \mathfrak{C}^1(t) \ \mbox{for}\ 0 \leq t \leq T^1, \\
		\mathfrak{C}^2(t - T^1) \ \mbox{for}\ T^1 < t \leq T^1 + T^2 \end{array} \right.
	\]	
	is a DMV solution of the Euler system in $(0, T_1 + T_2) \times \Omega$, with the initial data $(\vr^1_0, \vm^1_0, S^1_0)$, 
	\and $\mathcal{E}_0$.
	
\end{Lemma}

We are now ready to prove that the defect of a maximal solution vanishes for large times. 

\begin{Proposition} \label{saP1}
	
Let $(\vr, \vm, S)$ be a $\prec_{DiP}$ maximal solution in the solution set $\mathcal{U}[\vr_0, \vm_0, S_0; \mathcal{E}_0]$. 

Then 
\[
\mathcal{D}(\vr, \vm, S)(t+) = \mathcal{E}_0 - \intO{ E(\vr, \vm, S)(t+, \cdot) } 
\to 0 \ \mbox{as}\ t \to \infty.
\]

\end{Proposition}	

\begin{proof}
Suppose that 
\begin{equation} \label{pom1}
\mathcal{D}(\vr, \vm, S)(\tau+) = \mathcal{E}_0 - \intO{ E(\vr, \vm, S)(\tau+, \cdot) } = \delta > 0
\end{equation}
Our goal is to concatenate the solution with another one, starting from 
\[
\tvr_0 = \vr(\tau, \cdot),\ \tvm_0 = \vm(\tau, \cdot), 
\]
and 
\[
\tS_0 > S(\tau-, \cdot), 
\]
where $\tS_0$ can be chosen large satisfying 
\begin{equation} \label{pom2}
\intO{ E (\tvr_0, \tvm_0, \tS_0 ) } = \mathcal{E}_0. 
\end{equation}

To this end, we first identify the temperature $\vt(\tau-, \cdot)$ 
from the relation
\[
S(\tau-, \cdot) = \vr(\tau, \cdot) \Big( c_v \log (\vt(\tau-, \cdot)) - 
\log(\vr(\tau,\cdot)) \Big)\ \mbox{on the set}\ \{ \vr(\tau, \cdot) > 0\}.
\]
The total energy evaluated at the same time $\tau$ reads
\[
\intO{ \left[ \frac{1}{2} \mathds{1}_{\vr > 0} \frac{|\vm|^2}{\vr}
(\tau, \cdot) + c_v \vr \vt (\tau-, \cdot) \right]	} = \mathcal{E}_0 - 
\delta.
\]
Thus to keep \eqref{pom2} valid, we may choose 
\[
\tvt_0 = (1 +\lambda) \vt(\tau-, \cdot) , \ \mbox{where}\ 
\lambda c_v \intO{ \vr \vt(\tau-, \cdot) } = \delta,  
\]
meaning 
\begin{equation} \label{pom3}
\lambda = \delta c_v^{-1}  \left( \intO{ \vr \vt(\tau-, \cdot) } \right)^{-1} \geq \frac{\delta }{\mathcal{E}_0}.
\end{equation}

Computing the associated entropy, we get 
\begin{align}
\tS_0 &= \vr(\tau, \cdot) \Big[ c_v \log(\tvt_0) - \log(\vr(\tau, \cdot))\Big] = S(\tau-, \cdot) + \vr(\tau, \cdot) c_v \log \left( 1 + \lambda \right)\br  &\geq S(\tau-, \cdot) + \vr(\tau, \cdot) c_v \log \left( 1 + \frac{\delta}{\mathcal{E}_0} \right).
\nonumber
\end{align}
Consequently, 
\begin{align}
\intO{ \tS_0 } &\geq \intO{ S(\tau-, \cdot) } + c_v \log \left( 1 + \frac{\delta}{\mathcal{E}} \right) \intO{\vr (\tau, \cdot) } \br &= 
\intO{ S(\tau-, \cdot) } + c_v \log \left( 1 + \frac{\delta}{\mathcal{E}_0} \right) M_0,
\label{pom4}
\end{align}
where 
\[
M_0 = \intO{\vr_0 }
\]
denotes the total mass of the fluid.

Finally, suppose $(\vr, \vm, S)$ is $\prec_{DiP}$ maximal. In addition, suppose there is a sequence $\tau_n \to \infty$ such that \eqref{pom1} holds for $\tau = \tau_n$ and a fixed $\delta > 0$. Clearly, this leads to a contradiction with the maximality condition. Indeed we have 
\[
\intO{ {S}(t, \cdot) } \to \mathcal{S}_\infty \ \mbox{as}\ 
t \to \infty.
\]
Using \eqref{pom4} and the concatenation argument we construct a new DMV solution $(\tvr, \tvm, \tS)$ such that 
\[
(\tvr, \tvm, \tS) = (\vr, \vm, S) \ \mbox{for all}\ t \in [0, \tau_n), 
\]
and 
\[
\intO{ \tS(\tau_n+, \cdot) } \geq \intO{ S(\tau_n - , \cdot)} 
+ c_v \log \left( 1 + \frac{\delta}{\mathcal{E}_0} \right) M_0 > \mathcal{S}_\infty
\]
if $\tau_n$ is large enough in contrast with $\prec_{DiP}$ maximality of 
$(\vr, \vm, S)$. Consequently, a maximal solution satisfied
$\mathcal{D}(\varrho, \vm, S) \to 0, $ what we wanted to prove.
	
\end{proof}

Finally, we introduce the class of absolute entropy maximisers. 

\begin{Definition} [{\bf Absolute entropy maximiser}] \label{caD1}
Let $(\tvr,\tvm, \tS)$ be a dissipative solution belonging to the solution set $\mathcal{U}[\vr_0, \vm_0, S_0; \mathcal{E}_0]$. 
We say that $(\tvr, \tvm, \tS)$ is \emph{absolute entropy miximiser} if 
\[
(\vr, \vm, S) \prec_{DiP} (\tvr, \tvm, \tS)
\]	
for all $(\vr, \vm, S) \in \mathcal{U}[\vr_0, \vm_0, S_0; \mathcal{E}_0]$.

\end{Definition}	

As shown in \cite[Theorem 3.2]{FeiLukYu}, any absolute entropy maximiser 
is necessarily a weak solution of the Euler system, more specifically $\mathcal{D}(t) = 0$. 
Moreover, repeating the arguments of \cite[Section 5]{FeiJuLu} we easily 
observe that any set $\mathcal{U}[\vr_0, \vm_0, S_0; \mathcal{E}_0]$ contains at most one maximal entropy maximiser. 

Let us summarise the properties of the solutions selected in Step 1.

\begin{Theorem}[{\bf Properties of solutions selected in Step 1}] \label{caT1}
	
Let 
\[
\mathcal{U}_S[\vr_0, \vm_0, S_0; \mathcal{E}_0] = {\rm argmax}\, \mathcal{F}_S
\Big[ \mathcal{U}[\vr_0, \vm_0, S_0; \mathcal{E}_0] \Big].
\]
	
Then the following holds:

\begin{itemize}
	\item For any $(\vr_0, \vm_0, S_0; \mathcal{E}_0) \in X_D$, the set  
$\mathcal{U}_S[\vr_0, \vm_0, S_0; \mathcal{E}_0]$ is a non--empty bounded closed convex subset of the Banach space $L^q_\omega(0, \infty; L^q(\Omega; R^{d+2}))$, $1 \leq q \leq 
\frac{2 \gamma}{\gamma +1}$.	 
\item The mapping 
\[
\mathcal{U}_S: (\vr_0, \vm_0, S_0; \mathcal{E}_0) \in X_D \mapsto 
\mathcal{U}_S[ \vr_0, \vm_0, S_0; 
\mathcal{E}_0] \in {\rm closed}[L^q_\omega(0,\infty; L^q(\Omega; R^{d+2})]
\]
is Borel--measurable with respect to the Wijsman topology on 
${\rm closed}[L^q_\omega(0, \infty; L^q(\Omega; R^{d+2})]$.

\item The energy defect vanishes for large times:
\[
\mathcal{D}(\vr, \vm, S)(t\pm) \to 0 \ \mbox{as}\ t \to \infty
\]
for any $(\vr, \vm, S) \in \mathcal{U}_S[\vr_0, \vm_0, S_0; \mathcal{E}_0]$.
\item If, in addition, the set $\mathcal{U}_S[\vr_0, \vm_0, S_0; \mathcal{E}_0]$ 
contains an absolute entropy maximiser $(\tvr, \tvm, \tS)$, then 
$(\tvr, \tvm, \tS)$ is a weak solution of the Euler system, the solution set
$\mathcal{U}_S[\vr_0, \vm_0, S_0; \mathcal{E}_0]$ is a singleton, and 
\[
\mathcal{U}_S[\vr_0, \vm_0, S_0; \mathcal{E}_0 ] = \{(\tvr, \tvm, \tS)   \}.
\]

\end{itemize}

\end{Theorem}

\section{Step 2 - turbulent solutions}
\label{ds}

If the set $\mathcal{U}_S[\vr_0, \vm_0, S_0; \mathcal{E}_0]$ fails to be a singleton, 
we propose the second selection based on maximisation of the energy defect (turbulent energy). Specifically, we consider the functional
\[
\mathcal{F}_E(\vr, \vm, S) = \int_0^\infty \exp(-t) \left( \intO{ 
E(\vr, \vm, S) (t, \cdot)}	\right) \dt.
\]
We close the selection procedure by setting 
\[
\mathcal{U}_{ES}(\vr_0, \vm_0, S_0; \mathcal{E}_0) = 
{\rm argmin}\, \mathcal{F}_E \Big[ \mathcal{U}_S[\vr_0, \vm_0, S_0; \mathcal{E}_0] \Big]. 
\]
As the energy functional $E$ is strictly convex in its domain, 
$\mathcal{U}_{ES}(\vr_0, \vm_0, S_0; \mathcal{E}_0)$ is a singleton, which finishes the selection process.

\subsection{Measurability of the second selection}

Our next objective is to show the Borel measurability of the mapping
associated with the second selection, 
\begin{align} 
\mathcal{U}_{ES} : (\vr_0, \vm_0, S_0; \mathcal{E}_0) \in X_D \mapsto 
{\rm argmin}\, \mathcal{F}_E \left[ \mathcal{U}_S[\vr_0, \vm_0, S_0; \mathcal{E}_0] \Big]       \right]	
\in L^q_{\omega} ((0, \infty )L^q(\Omega; R^{d+2}) ).
\label{ca6}
\end{align}
	
We start by introducing a Moreau--Yosida regularisation of the 
functional $\mathcal{F}_E$ on the Banach space 
$L^q_\omega(0, \infty; L^q(\Omega; R^{d+2}))$ proposed by 
Bacho \cite{Bacho}:
\begin{align} 
	\mathcal{F}^\ep_E &(\vr, \vm, S ) \br &= 
	\inf_{(\tvr, \tvm, \widetilde{S}) \in L^q_{\omega}((0, \infty); L^q(\Omega; R^{d+2}))  } \left[ \frac{1}{\ep} 
	\| (\vr, \vm, S ) - (\tvr, \tvm, \widetilde{S}) \|_{L^q_{\omega}((0, \infty); L^q(\Omega; R^{d+2})  }^q + 
	\mathcal{F}_E (\tvr, \tvm, \widetilde{S})) \right]
\label{ca9}	
\end{align}
for $\ep > 0$. 
Next, we show the Borel measurability of the mapping 
\begin{align} 
	\mathcal{U}_{ES}^\ep : (\vr_0, \vm_0, S_0; \mathcal{E}_0) \in X_D \mapsto 
	{\rm argmin} \mathcal{F}_E^\ep \Big[ \mathcal{U}_S [\vr_0, \vm_0, S_0; \mathcal{E}_0]       \Big]	
	\in L^q_{\omega} ((0, \infty )L^q(\Omega; R^{d+2}) ).
	\label{caa6}
\end{align}
In accordance with Remark \ref{Rac1}, the Borel measurability claimed in
\eqref{caa6} will follow as long as we show continuity 
of the mapping 
\[
\mathcal{A} \in {\rm closed}[L^q_{\omega} ((0, \infty )L^q(\Omega; R^{d+2}) ) ] \to {\rm argmin}\, \mathcal{F}^\ep_E [\mathcal{A}] 
\]
on the family of closed convex sets in $L^q_{\omega} ((0, \infty )L^q(\Omega; R^{d+2}) )$
with respect to the Wijsman topology.
The space $L^q_\omega (0, \infty; L^q(\Omega; R^{d+2}))$, $1 < q < \infty,$
being a reflexive, separable, uniformly convex Banach space with a uniformly convex dual, the convergence in the Wijsman topology is equivalent to convergence in the Mosco topology $\mathcal{M}$, see Beer \cite[Sonntag-Attouch Theorem in Section 4]{Beer1}.
We recall that
\[
\mathcal{A}_n \toMo \mathcal{A}
\]
if the two following properties hold:
\begin{itemize}
	\item For any $(\vr, \vm, S) \in \mathcal{A}$, there exists
	$(\vr_n, \vm_n, S_n) \in \mathcal{A}_n$ such that
	\[
	(\vr_n, \vm_n, S_n) \to (\vr, \vm, S) \ \mbox{(strongly) in}\ L^q_\omega (0, \infty; L^q(\Omega; R^{d+2})).
	\]	
	\item If
	\[
	(\vr_{n_k}, \vm_{n_k}, S_{n_k}) \in \mathcal{A}_{n_k} \to (\vr, \vm, S)
	\ \mbox{weakly in}\ L^q_\omega (0, \infty; L^q(\Omega; R^{d+2}))
	\]
	then
	\[
	(\vr, \vm, S) \in \mathcal{A}.
	\]
\end{itemize}	

We claim the following result. 
\begin{Lemma} \label{caL2}
	Let $(\mathcal{A}_n)_{n=1}^\infty$ be a sequence of bounded 
	closed convex sets in 
	$L^q_\omega(0,\infty; L^q(\Omega; R^{d+2}))$
	such that 
	\[
	\mathcal{A}_n \toMo \mathcal{A}, 
	\]
	where $\mathcal{A}$ is a bounded closed convex. 
	
	Then 
	\[
	{\rm argmin}\, \mathcal{F}^\ep_E [\mathcal{A}_n] \to 
	{\rm argmin}\, \mathcal{F}^\ep_E [\mathcal{A}] 
	\ \mbox{in}\ L^q_\omega(0,\infty; L^q(\Omega; R^{d+2})) \ \mbox{as}\ 
	n \to \infty. 
	\] 	
	
\end{Lemma}	

\begin{proof}
	
	Denote 
	\[
	(\underline{\vr}^n, \underline{\vm}^n, \underline{S}^n) = 
	{\rm argmin}\, \mathcal{F}^\ep_E[\mathcal{A}_n],\ 
	(\underline{\vr}, \underline{\vm}, \underline{S}) = 
	{\rm argmin}\, \mathcal{F}^\ep_E[\mathcal{A}]. 
	\]	
	In accordance with the Mosco convergence, there exists a sequence 
	\[
	(\vr_n, \vm_n , S_n) \in \mathcal{A}_n \to 
	(\underline{\vr}, \underline{\vm}, \underline{S}) \ \mbox{in}\ 
	L^q_\omega(0,\infty; L^q(\Omega; R^{d+2})).
	\]	
	As $\mathcal{F}^\ep_E$ is continuous, we get 
	\begin{equation} \label{ca13}
		\mathcal{F}^\ep_E (\underline{\vr}^n, \underline{\vm}^n, \underline{S}^n) \leq \mathcal{F}^\ep_E (\vr_n, \vm_n , S_n)
		\to \mathcal{F}^\ep_E (\underline{\vr}, \underline{\vm}, \underline{S}) \ \mbox{as}\ n \to \infty.
	\end{equation}
	
	Using the properties of the Mosco convergence again, we can extract a subsequence (not) relabeled such that 
	\[
	(\underline{\vr}^n, \underline{\vm}^n, \underline{S}^n) 
	\to (\tvr, \tvm, \widetilde{S}) \in \mathcal{A} 
	\ \mbox{weakly in}\ L^q_\omega(0,T; L^q(\Omega; R^{d+2})), 
	\]
	where 
	\begin{equation} \label{ca14}
		\mathcal{F}^\ep_E	(\tvr, \tvm, \widetilde{S}) \leq 
		\liminf_{n \to \infty} \mathcal{F}^\ep_E (\underline{\vr}^n, \underline{\vm}^n, \underline{S}^n).
	\end{equation}	
	
	Comparing \eqref{ca13}, \eqref{ca14} we conclude 
	\[
	(\tvr, \tvm, \widetilde{S}) = (\underline{\vr}, \underline{\vm}, \underline{S}),
	\]
	\begin{equation} \label{ca15}
		(\underline{\vr}^n, \underline{\vm}^n, \underline{S}^n) 
		\to (\underline{\vr}, \underline{\vm}, \underline{S}) 
		\ \mbox{weakly in}\ L^q_\omega(0,\infty; L^q(\Omega; R^{d+2})),
	\end{equation}
	and
	\[
	\mathcal{F}^\ep_E(\underline{\vr}^n, \underline{\vm}^n, \underline{S}^n) 
	\to \mathcal{F}^\ep_E(\underline{\vr}, \underline{\vm}, \underline{S})
	\]
	as $n \to \infty$. As $\mathcal{F}^\ep_E$ is strictly convex, we obtain the desired strong convergence in 
	\eqref{ca15}.
\end{proof}

\begin{Remark} \label{caR1}
	
	As $\mathcal{F}^\ep_E$ are coercive, the hypothesis of boundedness of the convex sets can be omitted.
	
\end{Remark}

Having shown Borel measurability of the mapping $\mathcal{U}^\ep_{ES}$, we deduce the desired measurability of $\mathcal{U}_{ES}$ from the 
following result.

\begin{Lemma} \label{caL1}
Let $\mathcal{A}$ be a bounded convex closed subset of 
$L^q_{\omega}((0, \infty); L^q(\Omega; R^{d+2}))$ 
such that $\mathcal{F}_E|_\mathcal{A}$ is proper. 

Then 
\begin{align}
{\rm argmin}\, \mathcal{F}^\ep_E [\mathcal{A}] \to 
{\rm argmin}\, \mathcal{F}_E [\mathcal{A}]\ \mbox{in}\ 
L^q_{\omega}((0, \infty); L^q(\Omega; R^{d + 2}))
\ \mbox{as}\ \ep \to 0. 
\label{ca10}	
\end{align}

\end{Lemma}	
\begin{proof}
We have 
\begin{align}
0 \leq \mathcal{F}^\ep_E(\vr, \vm, S) \nearrow \mathcal{F}_E(\vr, \vm, S)
\ \mbox{for any}\ (\vr, \vm, S) \in \mathcal{A}. 
\nonumber
\end{align}	
In particular, 
\[
\min_{\mathcal{A}} \mathcal{F}^\ep_E \nearrow \underline{F} \leq 
\min_{\mathcal{A}} \mathcal{F}_E = 
\mathcal{F}_E (\underline{\vr}, \underline{\vm}, \underline{S}),  
\]
where we have denoted
\[
(\underline{\vr}, \underline{\vm}, \underline{S}) =
{\rm argmin}\, \mathcal{F}_E [\mathcal{A}] 
\]

Let us denote
\[
(\underline{\vr}^\ep, \underline{\vm}^\ep, \underline{S}^\ep) =
{\rm argmin}\, \mathcal{F}^\ep_E [\mathcal{A}].
\]
It follows from the strict convexity of the norm in $L^q_\omega$ that there exist a family 
$(\tvr^\ep, \tvm^\ep, \widetilde{S}^\ep)$ such that 
\begin{align}
\mathcal{F}_E(\underline{\vr}, \underline{\vm}, \underline{S} ) \geq 
\mathcal{F}^\ep_E (\underline{\vr}^\ep, \underline{\vm}^\ep, \underline{S}^\ep) = \frac{1}{\ep} 
\| (\underline{\vr}^\ep, \underline{\vm}^\ep, \underline{S}^\ep ) - (\tvr^\ep, \tvm^\ep, \widetilde{S}^\ep) \|^q_{L^q_\omega} + \mathcal{F}_E 
(\tvr^\ep, \tvm^\ep, \widetilde{S}^\ep) \geq 0.
\label{ca11}
\end{align}

Thus, passing to the limit for a suitable subsequence, we have 
\begin{equation} \label{ca12}
(\underline{\vr}^\ep, \underline{\vm}^\ep, \underline{S}^\ep ) \to (\tvr, \tvm, \widetilde{S}), \ 
(\tvr^\ep, \tvm^\ep, \widetilde{S}^\ep) \to 
(\tvr, \tvm, \widetilde{S}) \ \mbox{weakly in}\ L^q_\omega 
\ \mbox{as}\ \ep \to 0, 
\end{equation}
where, as $\mathcal{A}$ is convex, 
\[
(\tvr, \tvm, \widetilde{S}) \in \mathcal{A}. 
\] 

Letting $\ep \to 0$ in \eqref{ca11} using weak lower semi--continuity of $\mathcal{F}_E$, we deduce 
\begin{align}
	\mathcal{F}_E(\underline{\vr}, \underline{\vm}, \underline{S} ) \geq \limsup_{\ep \to 0} \mathcal{F}_E 
	(\tvr^\ep, \tvm^\ep, \widetilde{S}^\ep) \geq 
	\liminf_{\ep \to 0} \mathcal{F}_E
	(\tvr^\ep, \tvm^\ep, \widetilde{S}^\ep) \geq 
	\mathcal{F}_E 
	(\tvr, \tvm, \widetilde{S}) \geq \mathcal{F}_E(\underline{\vr}, \underline{\vm}, \underline{S} ).
	\nonumber
\end{align}
We conclude 
\[
(\tvr, \tvm, \widetilde{S}) =(\underline{\vr}, \underline{\vm}, \underline{S} ),
\]
and the desired conclusion follows from \eqref{ca12}.

\end{proof}	

We conclude that the mapping 
\[
\mathcal{U}_{ES}: (\vr_0, \vm_0, S_0; \mathcal{E}_0) \in X_D \mapsto 
{\rm argmin}\, \mathcal{F}_E \Big[ \mathcal{U}_S [\vr_0, \vm_0, S_0; \mathcal{E}_0] \Big]
\in L^q_\omega(0, \infty; L^q(\Omega; R^{d+2}))
\]
being a pointwise limit of Borel--measurable mappings is Borel--measurable. 

Let us summarise the properties of the solutions selected in Step 2. 

\begin{Theorem}[{\bf Properties of solutions selected in Step 2}] \label{caT2}
	
	Let 
	\[
	(\vr, \vm, S) = \mathcal{U}_{ES} (\vr_0, \vm_0, S_0, \mathcal{E}_0).
	\]
	
	Then the following holds:
	
	\begin{itemize}
	 
		\item The mapping 
		\[
		(\vr_0, \vm_0, S_0; \mathcal{E}_0) \in X_D \mapsto 
		\mathcal{U}_{ES}(\vr_0, \vm_0, S_0; \mathcal{E}_0) \in L^q_\omega(0,\infty; L^q(\Omega; R^{d+2}))
		\]
		is Borel--measurable.
		
		\item The energy defect (turbulent energy) vanishes for large times, 
		\[
		\mathcal{D}(\vr, \vm, S)(t\pm) \to 0 \ \mbox{as}\ t \to \infty.
		\]
		\item If $\mathcal{U}_S[\vr_0, \vm_0, S_0; \mathcal{E}_0]$ is not a singleton, then $(\vr, \vm, S)$ is a truly measure--valued solution, meaning 
		$\mathcal{D} (\vr, \vm, S)\ne 0$.
		
	\end{itemize}

\end{Theorem}

\section{Basic properties of the selected solution}
\label{conc}

We conclude by recalling two essential properties of the selected solutions.

\subsection{Semigroup property}
\label{COO1}

Repeating the arguments of \cite{BreFeiHof19C}, we may show that the solution mapping 
\[
\mathcal{S}: t \in [0, \infty) ,\ (\vr, \vm, S; \mathcal{E}_0) \in X_D \mapsto \Big[ \mathcal{U}_{ES} (\vr, \vm, S; \mathcal{E}_0)(t - ,\cdot) ; \mathcal{E}_0 \Big] 
\]
enjoys the semigroup property. Specifically, 
\begin{align}
\mathcal{S}(t + s; 	(\vr, \vm, S; \mathcal{E}_0) ) = 
\mathcal{S} \Big( s; \mathcal{S} (\vr, \vm, S; \mathcal{E}_0) (t, \cdot) \Big),\quad t,s \geq 0.
\label{conc1}	
\end{align}

\subsection{Almost continuity}
\label{COO2}

In view of the Borel measurability of the solution mapping 
\[
\mathcal{U}_{ES}: X_D \to L^q_{\omega}(0, \infty; L^q(\Omega; R^{d+2})),
\]
we may use an abstract version of Egorov (Lusin) theorem 
(cf. Wisniewski \cite{Wis}) to deduce the following conclusion:

Let $\mu$ be a complete Borel probability measure on $X_D$. Then for any $\ep > 0$, there exists a closed set $K \subset X_D$, $\mu[K] > 1-\ep$ such that 
\[
\mathcal{U}_{ES}|K \to L^q_{\omega}(0, \infty; L^q(\Omega; R^{d+2}))
\]
is continuous. As a matter of fact, Wisniewski \cite{Wis} proves a stronger result, namely there is a sequence of \emph{continuous} mappings 
$(\mathcal{U}^\ep_{ES})_{\ep > 0}$ such that
\[
\mathcal{U}^\ep_{ES} \to \mathcal{U}_{ES} \quad  \mu - \mbox{almost surely.}
\]

\subsection{Instantaneous values of selected solutions}

It is easy to see that the instantaneous values of dissipative solution belong to the set $X_D$, 
\[
\Big[ \mathcal{U}_{ES}(\vr_0, \vm_0, S_0; \mathcal{E}_0); \mathcal{E}_0 \Big] (\tau-, \cdot) \in X_D \ \ \mbox{for any}\ \tau > 0.
\]
Moreover, the entropy $S$ can be identified with a c\` agl\` ad mapping ranging in the space $W^{-\ell, 2}(\Omega)$. In particular, 
the entropy $S$ belongs to the Skorokhod space of  c\` agl\` ad mappings,
\[
S \in D([0, \infty); W^{-\ell,2}(\Omega)).
\]
We refer to \cite[Appendix A1]{FeiNovOpen} for basic properties of the space $D([0, \infty); W^{-\ell,2}(\Omega))$.

Similarly to \cite{FeiJuLu}, we can show Borel measurability of the mapping 
\begin{equation} \label{conc2}
(\vr_0, \vm_0, S_0; \mathcal{E}_0) \in X_D \mapsto \Big[ \mathcal{U}_{ES} (\vr_0, \vm_0, S_0; \mathcal{E}_0) (\tau - , \cdot); \mathcal{E}_0 \Big] \in X_D
\end{equation}
for any fixed $\tau$. In particular, the ``almost continuity'' discussed in the preceding section can be extended to the mapping \eqref{conc2} for any fixed 
$\tau > 0$. We leave the details to the interested reader.

\section{One-step selection process - maximal dissipativity}
\label{O}

The two-step selection process is probably satisfactory from the analysis point of view, but less convenient in numerical implementations. We propose 
a single-step selection procedure inspired by the Second law of thermodynamics, formulated in the celebrated statement by Clausius:
\smallskip

\centerline{\it The energy of the world is constant; its entropy tends to a maximum.}

\smallskip

\noindent The total energy of a dissipative solution is given by its initial value $\mathcal{E}_0$. The total mass 
\[
M_0 = \intO{ \vr_0 }
\]
is a constant of motion as well. We define the \emph{equilibrium density and temperature} 
\begin{equation} \label{O1}
\Ov{\vr} = \frac{1}{|\Omega|}\intO{ \vr_0 } = \frac{M_0}{|\Omega|},\ \intO{ \Ov{\vr} e(\Ov{\vr}, \Ov{\vt}) } = \mathcal{E}_0\ \Rightarrow \
\Ov{\vt} = \frac{\mathcal{E}_0}{c_v M_0}. 
\end{equation}
The total entropy $S$ expressed constitutively in terms of $\vr$ and the internal energy ${E}_i = c_v \vr \vt$ is a concave function, in particular 
\begin{equation} \label{O2}
S(\vr(t, \cdot), E_i(t, \cdot)) - \frac{\partial S(\Ov{\vr}, \Ov{\vt}) }{\partial \vr}(\vr - \Ov{\vr}) - 	\frac{\partial S(\Ov{\vr}, \Ov{\vt}) }{\partial E_i}(c_v \vr \vt - c_v \Ov{\vr} \Ov{\vt} ) - S(\Ov{\vr}, c_v \Ov{\vr} \Ov{\vt}) \leq 0.
\end{equation}	
Integrating \eqref{O2} and using \eqref{O1}, together with the inequality 
\[
\intO{ c_v \vr \vt } \leq \mathcal{E}_0, 
\]
we conclude 
\begin{equation} \label{O3}
\intO{ S(t \pm , \cdot) } \leq S(\Ov{\vr}, c_v \Ov{\vr} \Ov{\vt}) \ \mbox{for all}\ t 
\end{equation}	
for any dissipative solution of the Euler system. In other words, the entropy is maximised at the constant equilibrium state $\Ov{\vr}, \Ov{\vt}$ specified 	
in \eqref{O1}. 

Summarising the previous discussion, we conclude that a selection criterion compatible with the Second law of thermodynamics should minimise the 
distance between a dissipative solution $(\vr, \vm, S)$ and the constant equilibrium $(\Ov{\vr}, 0, S(\Ov{\vr}, \Ov{\vt}))$. To measure the distance, we 
use the Bregmann divergence associated to the total energy $E(\vr, \vm, S)$, namely  
\begin{align} 
E &\left( \vr, \vm, S \Big| \Ov{\vr}, 0, \Ov{\vt} \right) \br &= 
E(\vr, \vm, S) -  \left( c_v \Ov{\vt} - \Ov{\vt} s (\Ov{\vr}, \Ov{\vt}) + \frac{p (\Ov{\vr}, \Ov{\vt} )}{\Ov{\vr}} \right)(\vr - \Ov{\vr}) 
- \Ov{\vt} \Big(S - S(\Ov{\vr}, \Ov{\vt}) \Big) - \Ov{\vr} e(\Ov{\vr}, \Ov{\vt}),
\label{O4} 	
\end{align}	
see \cite[Chapter 4, Section 4.1.6]{FeLMMiSh}. 

Relation \eqref{O4} integrated over $\Omega$ yields 
\begin{equation} 
	\intO{ E \left( \vr, \vm, S \Big| \Ov{\vr}, 0, \Ov{\vt}\right) } = 
	 \intO{ E(\vr, \vm, S) } 
	- \Ov{\vt} \intO{ \Big(S - S(\Ov{\vr}, \Ov{\vt}) \Big) } - \Ov{\vr} e(\Ov{\vr}, \Ov{\vt})|\Omega|.
	\label{O5} 	
\end{equation}	
This motivates the definition of a cost functional in the form 
\begin{equation} \label{O6}
\mathcal{F}_{D} [\vr, \vm, S; \mathcal{E}_0] = \int_0^\infty \exp(-t) \left( \intO{ \left( E(\vr, \vm, S) - \Ov{\vt} S \right) } \right) \dt,\  
\quad  \Ov{\vt} = \frac{\mathcal{E}_0}{c_v M_0 }.
\end{equation}	
The cost functional being strictly convex, the minimisation procedure yields a unique solution in only one step:
\begin{equation} \label{O7}
\mathcal{U}_D(\vr, \vm, S; \mathcal{E}_0) = {\rm argmin}\, \mathcal{F}_{D} \Big[ \mathcal{U}[\vr, \vm, S; \mathcal{E}_0] \Big]. 
\end{equation}
Borel measurability of the solution mapping can be shown exactly as in Section \ref{ds}. We call the quantity 
\[
\mathcal{U}_D(\vr, \vm, S; \mathcal{E}_0); \mathcal{E}_0 
\] 
\emph{maximal dissipative solution of the Euler system}. 

We have shown the following result. 

\begin{Theorem}[{\bf Maximal dissipative solution}] \label{OT1}
For any initial data $(\vr_0, \vm_0, S_0; \mathcal{E}_0) \in X_D$, the Euler system admits a unique maximal dissipative solution 
\[
\mathcal{U}_D(\vr, \vm, S; \mathcal{E}_0)
\]	
given through formula \eqref{O7}. The mapping 
\[
(\vr_0, \vm_0, S_0; \mathcal{E}_0) \in X_D \mapsto \mathcal{U}_D(\vr, \vm, S; \mathcal{E}_0) \in L^q_{\omega}(0, \infty; L^q(\Omega; R^{d+2}))
\]
is Borel--measurable. The solution mapping enjoys the properties of an almost continuous semigroup specified in Sections 
\ref{COO1}, \ref{COO2}.
\end{Theorem}

\section{Conclusion}
Recent breakthrough results on the non-uniqueness of admissible (entropy) weak solutions for the Euler equations of gas dynamics have initiated intensive discussions on the possibility of recovering well-posedness through a suitable selection process. In this paper, we have proposed two selection strategies.  

The first one is a two-step selection: in the first step, the total entropy 
\[
\int_0^\infty \exp(-t) \left( \intO{ S(t, \cdot) } \right) \dt
\]
is maximised in the set of generalised, dissipative measure--valued solutions. If the first step failed to identify a unique solution, we proceed to the second step. Then, in the set selected in the first step, we search for a solution  that minimises the mean energy \[
\int_0^\infty \exp(-t) \left( \intO{ 
E(\vr, \vm, S) (t, \cdot)}	\right) \dt.
\]
As the latter is a strictly convex functional in its domain, we have selected a unique solution of the Euler equations.  

We have observed that if the selected solution is a 
weak solution of the Euler system, then it must have been selected in the first step.
The solution  selected in the second step is \emph{a unique turbulent solution, i.e.~a truly measure--valued solution.}

The second selection strategy combines directly the above two selection steps yielding  a one-step selection: 
\[
\int_0^\infty \exp(-t) \left( \intO{ \left( E(\vr, \vm, S) - \Ov{\vt} S \right) } \right) \dt \quad \mbox{with}\ \Ov{\vt} = \frac{\mathcal{E}_0}{c_v M_0 },
\]
where $\mathcal{E}_0$ is the initial total  energy and $M_0$ the total mass. We have proved that  our selection strategies yield 
Borel--measurable initial data $\mapsto$ solution mappings. Moreover, the  solution mappings are almost continuous and enjoy the semigroup property. In the future, we aim to design a numerical method that approximates the unique solution selected by the proposed selection strategies.

\section*{Compliance with Ethical Standards}\label{conflicts}

\small
\par\noindent

{\bf Author contribution declaration statement}.

Eduard Feireisl: Writing – review \& editing, Investigation.

M\' aria Luk\'a\v{c}ov\'a-Medvi\softd ov\'a : Writing – review \& editing, Investigation. 

\smallskip
\par\noindent 

{\bf Conflict of Interest}. The authors declare that they have no conflict of interest.

\smallskip
\par\noindent

{\bf Data Availability}. Data sharing is not applicable to this article as no datasets were generated or analysed during the current study.


\def\cprime{$'$} \def\ocirc#1{\ifmmode\setbox0=\hbox{$#1$}\dimen0=\ht0
	\advance\dimen0 by1pt\rlap{\hbox to\wd0{\hss\raise\dimen0
			\hbox{\hskip.2em$\scriptscriptstyle\circ$}\hss}}#1\else {\accent"17 #1}\fi}

\end{document}